\documentclass[12pt]{amsart}
\usepackage{amsfonts,amsmath,amssymb,fullpage,graphicx, comment}
\usepackage{verbatim}
\usepackage{xcolor}
\def\frak{\mathfrak}
\newtheorem{thm}{Theorem}[section]
\newtheorem{cor}[thm]{Corollary}
\newtheorem{lem}[thm]{Lemma}

\theoremstyle{definition}

\theoremstyle{remark}
\newtheorem{rem}[thm]{\bf{Remark}}
\numberwithin{equation}{section}

\newcommand{\beas}{\begin{eqnarray*}}
\newcommand{\eeas}{\end{eqnarray*}}
\newcommand{\bes} {\begin{equation*}}
\newcommand{\ees} {\end{equation*}}
\newcommand{\be} {\begin{equation}}
\newcommand{\ee} {\end{equation}}
\newcommand{\bea} {\begin{eqnarray}}
\newcommand{\eea} {\end{eqnarray}}
\newcommand{\ra} {\rightarrow}

\newcommand{\R}{\mathbb R}

\newcommand{\C}{\mathbb C}

\newcommand{\N}{\mathbb N}

\newcommand{\la}{\lambda}

\begin{document}

\title[Chernoff's theorem on symmetric spaces] {A theorem of Chernoff on quasi-analytic functions for Riemannian symmetric spaces}

\author{Mithun Bhowmik, Sanjoy Pusti and Swagato K. Ray}

\address{(Mithun Bhowmik) Department of Mathematics, IISc Bangalore-560012, India}
\email{mithunb@iisc.ac.in}

\address{(Sanjoy Pusti) Department of Mathematics, IIT Bombay, Powai, Mumbai-400076, India}
\email{sanjoy@math.iitb.ac.in}

\address{(Swagato K Ray) Stat-Math Unit, Indian Statistical Institute, Kolkata-700108, India}
\email{swagato@isical.ac.in}

\thanks{The first author is supported by the Department of Science and Technology, India (INSPIRE Faculty Award).}


\begin{abstract}
An $L^2$ version of the classical Denjoy-Carleman theorem regarding quasi-analytic functions was proved by P. Chernoff on $\R^n$ using iterates of the Laplacian. We give a simple proof of this theorem which  generalizes the result on $\R^n$ for any $p\in [1, 2]$. We then extend this result to Riemannian symmetric spaces of compact and noncompact type for $K$-biinvariant functions.
\end{abstract}

\subjclass[2010]{Primary 43A85; Secondary 22E30, 33C67}

\keywords{Riemannian symmetric space, Quasi-analyticity,  Chernoff's theorem}

\maketitle
\section{Introduction}
A quasi-analytic class of functions is a generalization of the class of real analytic functions with respect to the following well-known property: if f is an analytic function on $(a,b)\subset\R$ then vanishing of $f$ along with all its derivatives at some point $x_0\in (a,b)$ implies that $f$ vanishes identically on $(a,b)$. Quasi-analytic classes are larger classes of functions for which this property still holds.
The class of quasi-analytic functions  are well-known on $\R$ and their characterization was given by the celebrated Denjoy-Carleman theorem \cite[Theorem 19.11]{Ru}. There are several generalizations and extensions of this result to several variables and certain Riemannian manifolds \cite{B, BT}. However, all these results deal with the $L^{\infty}$-norms of the derivatives involved. It is also well-known that quasi-analyticity of suitable functions can also be characterized in terms of the Fourier transform or Fourier coefficients of the functions. The earliest results in this direction go back to the works of Ch. de la Vall\'ee Poussin and A. E. Ingham \cite{Po, In}. Some recent research regarding generalizations of these latter results (see \cite{BPR, BGST, GT1, GT2}) brings out the importance of an interesting variant of the Denjoy Carleman theorem proved by Chernoff in \cite{Ch}. While discussing the notion of quasi-analytic vectors on Hilbert spaces, Chernoff proved the following variant of the Denjoy-Carleman theorem for functions on $\R^n$ using the $L^2$-norm of the iterates of the Laplacian $\Delta_{\R^n}$, instead of the $L^{\infty}$-norm.
\begin{thm}[\cite{Ch}, Theorem 6.1]\label{thm-intro}
Let $f:\R^n\to\C$ be a smooth function such that
for all $m\in \N\cup \{0\}$, $\Delta_{\R^n}^m f\in L^2(\R^n)$, and
\be \label{Carl-cond-Rn}
\sum_{m\in \N}\|\Delta_{\R^n}^m f\|_2^{-\frac{1}{2m}}=\infty.
\ee
If there exists $x_0\in \R^n$, such that $f$ and all its partial derivatives $\partial^{\alpha}f=\frac{\partial^{\alpha_1+\ldots\alpha_n}f}{\partial x_1^{\alpha_1}\ldots\partial x_n^{\alpha_n}}$, $\alpha\in(\N\cup \{0\})^n$, vanish at $x_0$ then $f$ vanishes identically.
\end{thm}
The following analogue of Chernoff's result was recently proved by the authors for Riemannian symmetric space $G/K$ of noncompact type involving iterates of the Laplace-Beltrami operator $\Delta$, where $G$ is a connected noncompact semisimple Lie group with finite center and $K$ is a maximal compact subgroup.
\begin{thm}[\cite{BPR}]\label{thm-cher-X-weaker}
Let $f\in C^\infty(G/K)$ be such that $\Delta^m f\in L^2(G/K)$, for all $m\in \N\cup \{0\}$ and
\be \label{Carl-cond-X-weak}
\sum_{m\in\N}\|\Delta^mf\|_2^{-\frac{1}{2m}}=\infty.
\ee
If $f$ vanishes on a nonempty open set in $G/K$ then $f$ vanishes identically.
\end{thm}

As is evident from the above statement that the vanishing condition on the function $f$ is much stronger here in comparison to Theorem \ref{thm-intro}. Similar results under analogous conditions of vanishing of the function on nonempty open sets have recently been proved in a wide variety of situations  \cite{GT1, GT2}. In \cite{BGST} an analogue of Theorem \ref{thm-intro} for the Heisenberg groups was proved under the condition of vanishing of the function and its partial derivatives at a single point albeit for a restricted class of $L^2$ functions.
One of the interesting problems which comes in the way of extending Theorem \ref{thm-intro} to the setup of certain homogeneous spaces of Lie groups, is to do with the choice of differential operators which are required to annihilate the given function at a particular point $x_0$. Indeed, in the setting of $G/K$, the vanishing of $G$-invariant differential operators applied to a function at a point is analogous to the vanishing of
partial derivatives of a function at a point of $\R^n$. For the rank one symmetric spaces of noncompact type one
knows that such differential operators are polynomials in the Laplace-Beltrami operator. On the other hand, it was observed by Chernoff that Theorem \ref{thm-intro} is false under the weaker assumption of vanishing of $\Delta_{\R^n}^m f(x_0)$, for all $m\in \N \cup \{0\}$ instead of vanishing of all partial derivatives of $f$ at $x_0$. Likewise, in \cite[Example 3.7]{BPR} we have shown that an exact analogue of Chernoff's result is not true for Riemannian symmetric spaces $G/K$ of noncompct type if we restrict ourselves only to the class of $G$-invariant differential operators. This is one of the reasons why a much stronger  hypothesis of vanishing of the function on an open set was used in Theorem \ref{thm-cher-X-weaker}.

In this paper we first suggest an alternative way of proving Theorem \ref{thm-intro} using an important result of M. de Jeu (see Lemma \ref{lempolydense}). This alternative approach is motivated by the observation that in Theorem \ref{thm-intro} if we restrict ourselves only to the class of radial functions then it is possible to prove the result under the assumption that $\Delta_{\R^n}^m f(0)$ vanishes for all $m\in \N \cup \{0\}$. We refer the reader to Remark \ref{counter1} for more detailed discussion in this regard. This alternative approach, in fact, helps us to prove the following generalization of Theorem \ref{thm-intro}.

\begin{thm}
\label{thm-cher-Rn}
Let $f:\R^n\to\C$ be a smooth function and $p\in[1, 2]$. Suppose that for all $m\in \N\cup \{0\}$, $\Delta_{\R^n}^m f\in L^p(\R^n)$, and
\be \label{Carl-cond-Rn}
\sum_{m\in \N}\|\Delta_{\R^n}^m f\|_p^{-\frac{1}{2m}}=\infty.
\ee
If there exists $x_0\in \R^n$, such that $f$ and all its partial derivatives $\partial^{\alpha}f=\frac{\partial^{\alpha_1+\ldots\alpha_n}f}{\partial x_1^{\alpha_1}\ldots\partial x_n^{\alpha_n}}$, $\alpha\in(\N\cup \{0\})^n$, vanish at $x_0$ then $f$ vanishes identically.
\end{thm}

We will use the idea of this proof to prove analogues of Chernoff's theorem for Riemannian symmetric spaces of compact and noncompact type, with vanishing condition at a single point but for a restricted class of functions which are analogous to radial functions.
In the context of Riemannian symmetric spaces of noncompact and compact type we prove the following analogues of Theorem \ref{thm-cher-Rn}. Let $\Delta$ be the Laplace-Beltrami operator on the Riemannian symmetric space $G/K$ of noncompact type and $\tilde\Delta$ be that on symmetric space $U/K$ of compact type. We refer the reader to sections $3$ and $4$ for the meaning of symbols.

\begin{thm} \label{thm-cher-X}
Let $f\in C^\infty(G//K)$ and $p\in [1,2]$. Suppose $\Delta^m f\in L^p(G//K)$, for all $m\in \mathbb N\cup \{0\}$ and
\be \label{Carl-cond-X}
\sum_{m\in \mathbb N} \|\Delta^m f\|_p^{-\frac{1}{2m}} =\infty.
\ee
If there exists $x_0\in G/K$, such that $Df(x_0)=0$, for all $D\in {\bf D}(G/K)$ then f is identically zero.
\end{thm}

This theorem is false if $p>2$ (see Remark \ref{counter1}).

\begin{thm} \label{thm-cher-com}
Suppose $f\in C^\infty(U//K)$ satisfies the condition
\be \label{Carl-cond-com}
\sum_{m\in \mathbb N} \|\tilde{\Delta}^m f\|_p^{-\frac{1}{2m}} =\infty,
\ee
for some $p\in [1,\infty]$. If $Df$ vanishes at the identity coset $o$ for all $D\in {\bf D}(U/K)$, then f vanishes identically.
\end{thm}

This theorem may not hold true if the identity coset $o$ is replaced by a different coset $x_0K$ (see Remark \ref{counter2}).
The proof of Theorem \ref{thm-intro} depends heavily on the theory of unbounded self adjoint operators on Hilbert spaces and it also uses the structure of dilation which is available on $\R^n$. The idea of this proof does not seem to work even for the case $p=2$ to obtain Theorem \ref{thm-cher-X} and Theorem \ref{thm-cher-com}  due to the lack of dilation structure in the context of Riemannian symmetric spaces. 
However, the alternative proof of Theorem \ref{thm-intro} which we have suggested uses a different technique altogether. 
The main idea behind this approach is to suitably use the connection between the Carleman type condition (\ref{Carl-cond-Rn})- (\ref{Carl-cond-com}) and a result regarding polynomial approximation proved in \cite{Dj}. This same idea then works for all $p\in [1,2]$. It is this method which we have employed to prove Theorem \ref{thm-cher-X} and Theorem \ref{thm-cher-com}.
Whether this method can be suitably modified to prove an analogue of Theorem \ref{thm-cher-X} for functions which are not $K$-biinvariant is an open question. However, it seems to us that to prove an exact analogue of Theorem \ref{thm-cher-Rn} for functions on $G/K$, it is perhaps necessary to consider a larger class of differential operators which are not necessarily $G$-invariant. For a discussion on quasi-analyticity and polynomial approximations on Lie groups, we refer the reader to \cite{DJ2} and the references therein.

This paper is organized as follows. In section $2$ we prove Theorem \ref{thm-cher-Rn}. In section $3$ and section $4$ we prove Theorem \ref{thm-cher-X} and Theorem \ref{thm-cher-com} respectively.




\section{Chernoff's Theorem for Euclidean spaces}
In this section we prove Theorem \ref{thm-cher-Rn}. The proof depends on the following result from approximation theory.
\begin{lem}[\cite{Dj}] \label{lempolydense}
Let $\mu$ be a finite Borel measure on $\R^n$ such that, for all $m\in \N$ and $1\leq j\leq n$ the quantities $S_j(m)$, defined by
\be \label{defn-Sj}
S_j(m)=\int_{\R^n}|\la_j|^m ~ d\mu(\la),
\ee
are finite. If for each $j\in \{1, \cdots, n\}$, the sequence $\{S_j(2m)\}_{m=1}^{\infty}$ satisfies the Carleman's condition
\be \label{carlcond}
\sum_{m\in \N} S_j(2m)^{-\frac{1}{2m}}= \infty,
\ee
then the polynomials constitute a dense subspace of $L^1(\R^n, d\mu)$.
\end{lem}
In the following the Fourier transform of $f\in L^1(\R^n)$ is defined by the usual formula
\be
\mathcal F{f}(\la)=\int_{\R^n}f(x)e^{-2\pi i\la \cdot x }dx.\nonumber
\ee
\begin{proof}[Proof of Theorem \ref{thm-cher-Rn}]
Using translation invariance of $\Delta_{\R^n}$, we may assume without loss of generality that $f\in C^\infty(\R^n)$, satisfying (\ref{Carl-cond-Rn}), vanishes along with $\partial^\alpha f$ at the origin for all $\alpha\in \left(\N\cup\{0\}\right)^n$. Using $\mathcal F{f}$ we define a measure $\mu$ on $\R^n$ by
\bes
\mu(E)=\int_E|\mathcal F{f}(\la)|~d\la,
\ees
for all Borel subsets $E$ of $\R^n$. We show that $\mu$ is a finite measure and the polynomials in $n$-variables are contained in $L^1(\R^n, \mu)$. We first work out the case $p\in (1, 2]$. Using the hypothesis $\Delta_{\R^n}^mf\in L^p(\R^n)$, for all $m\in \N\cup \{0\}$, and applying H\"older's and Hausdorff-Young inequalities we get that for $k\in \N\cup \{0\}$
\beas
\int_{\R^n} |\lambda|^k |\mathcal Ff(\lambda)|~d\lambda &<& \int_{\R^n}\left(1+|\lambda|^2\right)^{k+n} |\mathcal Ff(\lambda)|\left(1+|\lambda|^2\right)^{-n}~d\lambda\\
&\leq& \left(\int_{\R^n}\left(1+|\lambda|^2\right)^{(k+n)p^\prime} |\mathcal Ff(\lambda)|^{p^\prime}~d\lambda\right)^{1/{p^\prime}}~\left(\int_{\R^n} \left(1+|\lambda|^2\right)^{-np}~d\lambda\right)^{1/{p}}\\
&\leq & \|\left(1+\Delta_{\R^n}\right)^{k+n}f\|_{p} \left(\int_{\R^n}\left(1+|\lambda|^2\right)^{-np}~d\lambda\right)^{1/{p}}<\infty.
\eeas
For the case $p=1$, we have
\beas
\int_{\R^n} |\lambda|^k |\mathcal Ff(\lambda)|~d\lambda &\leq& \sup_{\lambda\in \R^n} \left\{\left(1+|\lambda|^2\right)^{(k+n)} |\mathcal Ff(\lambda)| \right\}~\int_{\R^n} \left(1+|\lambda|^2\right)^{-n}~d\lambda\\
&\leq & \|\left(1+\Delta_{\R^n}\right)^{k+n}f\|_{1} \int_{\R^n}\left(1+|\lambda|^2\right)^{-n}~d\lambda <\infty.
\eeas
Hence, for $j\in\{1, \cdots, n\}, m\in \N$, the  moment sequence $S_j(m)$ described in Lemma \ref{lempolydense} are well defined and $\mathcal Ff\in L^1 (\R^n)$. It also follows that the polynomials in $n$-variables are contained in $L^1(\R^n, \mu)$. Using the condition (\ref{Carl-cond-Rn}) we now show that the moment sequence $S_j(2m)$ satisfies the Carleman's condition (\ref{carlcond}) for each $j\in \{1, \cdots, n\}$. For $m\in \N$, we have
\bea \label{Sj-est-0}
S_j(2m) &\leq& \int_{\{\la\in \R^n:|\la|<1\}}|\la|^{2m}~|\mathcal F f(\la)|~~d\la + \int_{\{\la\in \R^n:|\la|\geq 1\}}|\la|^{2m}~|\mathcal F f(\la)|~~d\la.
\eea
If the support of $\mathcal F f$ is contained in $B(0, 1)$, then the second integral in the right-hand side is zero. Hence,
\bes
S_j(2m)\leq \int_{\{\la\in \R^n:|\la|<1\}}~|\mathcal F f(\la)|~~d\la=\|\mathcal Ff\|_{1}.
\ees
Therefore, in this case $S_j(2m)$ satisfies the Carleman's condition (\ref{carlcond}).  We now consider the case when the support of $\mathcal F f$ is not contained in $B(0, 1)$. Once again we assume that $p\in (1, 2]$. Using H\"older's inequality and Hausdorff-Young inequality it follows from (\ref{Sj-est-0}) that
\beas
S_j(2m) &\leq& |B(0, 1)|^{1/p} \left(\int_{\{\la\in \R^n:|\lambda|<1\}}|\la|^{2mp^\prime}~|\mathcal F f(\la)|^{p^\prime}~d\la\right)^{1/p^\prime}\\
&& + \left(\int_{\{\la\in \R^n:|\la|\geq 1\}}|\la|^{2(m+n)p^\prime}~|\mathcal F f(\la)|^{p^\prime}~d\la \right)^{\frac{1}{p^\prime}}~ \left(\int_{\{\la\in \R^n:|\la\geq 1|\}} \frac{1}{|\la|^{2np}}~d\la\right)^{\frac{1}{p}}\\
&\leq& |B(0, 1)|^{1/p} \|f\|_p + A_{n,p} \|\Delta_{\R^n}^{m+n} f\|_p,
\eeas
where the constant $A_{n, p}$ depends only on the dimension $n$ and $p$. We observe that $\|\Delta_{\R^n}^{m+n}f\|_p$ is bounded below by a constant $B_{n, p}=\|\chi_{|\la |\geq 1}\mathcal F f \|_p$, which is independent of $m$. Since the support of $\mathcal Ff$ is not contained in $B(0, 1)$, it follows that $B_{n, p}>0$. Therefore, we have from the inequality above that there exists a positive constant $C_{n, p}$ independent of $m$ such that
\bes
S_j(2m) \leq C_{n, p} \|\Delta_{\R^n}^{m+n} f\|_p,\:\:\:\:\: C_{n, p}=|B(0,1)|^{1/p}\|f\|_pB_{n,p}^{-1}+ A_{n, p}.
\ees
Consequently
\bes
|S_j(2m)|^{-\frac{1}{2m}} \geq C_{n,p}^{-\frac{1}{2m}}\|\Delta_{\R^n}^{m+n}f\|_p^{-\frac{1}{2m}} = C_{n, p}^{-\frac{1}{2m}} ~ \left(\|\Delta_{\R^n}^{m+n}f\|_p^{-\frac{1}{2(m+n)}}\right)^{\left(1+\frac{n}{m}\right)}.
\ees
Similar computation shows that the above estimate is also valid for $p=1$.
We now use the following fact \cite[Lemma 3.3]{BPR}: let $\{a_m\}$ be a sequence of positive numbers such that the series $\sum_{m\in \N}a_m$ diverges. Then for any given $n\in\N$, the series $\sum_{m\in \N} a_m^{1+\frac{n}{m}}$ diverges. This fact together with the hypothesis (\ref{Carl-cond-Rn}) implies that
\bes
\sum_{m=1}^{\infty} S_j(2m)^{-\frac{1}{2m}}=\infty.
\ees
Lemma \ref{lempolydense} now implies that the polynomials form a dense subspace of $L^1(\R^n, \mu)$. Since the Fourier transforms of $\partial^\alpha f$ are integrable for all $\alpha\in \left(\N\cup \{0\}\right)^n$, by the Fourier inversion formula it follows that
\bes
(\partial^\alpha f)(x) = \int_{\R^n} (2\pi i\lambda_1)^{\alpha_1}\cdots (2\pi i\lambda_n)^{\alpha_n} \mathcal F f(\la) ~e^{2\pi i x\cdot \lambda} d\la, \:\:\:\: x\in\R^n,\:\alpha\in \N\cup \{0\}.
\ees
The vanishing of the quantities $\partial^{\alpha}f(0)$, for all $\alpha\in \left(\N\cup\{0\}\right)^n$, now implies that
\be \label{cond-2}
\int_{\R^n} P(\lambda)\mathcal F f(\la) ~ d\la=0,
\ee
for all polynomials $P$. We observe that $\mathcal Ff\in L^{p^\prime}(\R^n)\cap L^1(\R^n)$ and hence is in $L^2(\R^n)$. Therefore, $\overline{\mathcal F f} \in L^1(\R^n, d\mu)$ and we can approximate $\overline{\mathcal F f}$ by polynomials $P$, that is, for any given $\epsilon>0$ there exists a polynomial $P_\epsilon$ such that
\bes
\|\overline{\mathcal F f}- P_\epsilon\|_{L^1(\R^n, d\mu)}< \epsilon.
\ees
It now follows that
\beas
&& \int_{\R^n}|\mathcal F f(\la)|^2~d\la =   \int_{\R^n}\overline{\mathcal F f(\la)}\mathcal F f(\la)~d\la =\left| \int_{\R^n}\left(\overline{\mathcal F f(\la)}-P_\epsilon(\la)+P_\epsilon(\la)\right)~\mathcal F f(\la)~d\la \right|\\
&\leq & \int_{\R^n}|\overline{\mathcal F f(\la)}-P_\epsilon(\la)|~d\mu(\lambda) + \left|\int_{\R^n} \mathcal F{f}(\la)P_\epsilon(\la)d\la \right| < \epsilon,
\eeas
the second integral being zero by (\ref{cond-2}). Consequently, $\mathcal F f$ is zero and hence so is $f$.	This completes the proof
\end{proof}
\begin{rem}
It is not known to us at the moment whether Theorem \ref{thm-cher-Rn} remains true for $p>2$.
\end{rem}

\section{Chernoff's theorem for symmetric spaces of noncompact type}

In this section, we first review briefly  the necessary preliminaries regarding semisimple Lie groups and harmonic analysis on Riemannian symmetric spaces. These are standard
and can be found, for example, in \cite{GV, H, H1, H2}. To make the article self-contained, we shall gather only those results which will be used throughout this paper.

Let $G$ be a semisimple Lie group, connected, noncompact, with finite center, and $K$ be a maximal compact subgroup of $G$. The homogeneous space
$X = G/K$ is a Riemannian symmetric space of noncompact type. Let $\mathfrak g = \mathfrak k \oplus \mathfrak p$ be
the Cartan decomposition of the Lie algebra of $G$. There is a natural identification
between $\mathfrak p$ and the tangent space of $X$ at the origin. The Killing form of $\mathfrak g$ induces
a $K$-invariant inner product on $\mathfrak p$, hence a $G$-invariant Riemannian metric on $X$.

We fix a maximal abelian subspace $\mathfrak a$ in $\mathfrak p$. The rank of $X$ is the dimension $l$ of $\mathfrak a$. We shall identify $\mathfrak a$ endowed with
the inner product induced from $\mathfrak p$ with $\R^l$ and let $\mathfrak a^\ast$ be the real dual of $\mathfrak a$.  Let $\Sigma \subset \mathfrak a^\ast$ be the root system of $(\mathfrak g, \mathfrak a)$. Let $M'$ and $M$ be the normalizer and
centralizer of $\mathfrak a$ in $K$ respectively. Then $M$ is a normal subgroup of $M'$ and normalizes $N$.
The quotient group $W = M'/M$ is a finite group, called the Weyl group of the pair $(\mathfrak g, \mathfrak k)$.
$W$ acts on $\mathfrak a$ by the adjoint action. Once a positive Weyl chamber $\mathfrak a^+\subset \mathfrak a$ has been selected, let $\Sigma^+$ denote the corresponding set of positive
roots. Let $n$ be the dimension of $X$, that is, $n=l+ \sum_{\alpha\in \Sigma^+} m_\alpha$, where $m_\alpha$ is the dimension of the positive root subspace $\mathfrak g_\alpha$. Let $\rho\in \mathfrak a^\ast$ denote the half sum of all positive roots counted with their multiplicities $m_\alpha$:
\bes
\rho= \frac{1}{2} \sum_{\alpha\in \Sigma^+} m_\alpha \alpha.
\ees
We extend the inner product on $\mathfrak{a}$ induced by $B$ to $\mathfrak{a}^*$ by duality.
The elements of the Weyl group $W$ acts on $\mathfrak a^*$ by the formula $sY_{\la}=Y_{s\la}$, for $s\in W, \la\in\mathfrak a^*$. Let $\mathfrak{a}_\C^*$ denote the complexification of $\mathfrak{a}^*$, that is, the set of all complex-valued real linear functionals on $\mathfrak{a}$. Let $\mathfrak n$ be the nilpotent Lie subalgebra of $\mathfrak g$ associated to $\Sigma^+$ and let $N = \exp \mathfrak n$
be the corresponding Lie subgroup of $G$. We have the decompositions
\beas
&& G= N(\exp \mathfrak a)K, \:\:\:\:\:\: (\textit{Iwasawa})\\
&& G=K(\exp{\overline{\mathfrak a^+}}) K, \:\:\:\: (\textit{Cartan})
\eeas

A function on $G$ is called $K$-biinvariant if 
\be
f(k_1 g k_2)=f(g),\:\:\:\:\text{for all $g\in G$, $k_1, k_2\in K$}.\nonumber
\ee
Using the polar decomposition of $G$ we may view an integrable or a continuous $K$-biinvariant function $f$ on $G$ as a function on $A_+$, or by using the inverse exponential map we may also view $f$ as a function on $\mathfrak{a}$ solely determined by its values on $\mathfrak{a}_+$. Henceforth, we shall denote the set of $K$-biinvariant functions in $L^p(G)$ by $L^p(G//K)$.
If $f\in L^1(G//K)$ then the spherical Fourier transform $\widehat{f}$ is defined by
\be \label{hlsphreln}
\widehat{f}(\la )= \int_Gf(g)\phi_{-\la}(g)~dg,
\ee
where
\be \label{philambda}
\phi_\la(g)
= \int_K e^{-(i\la+ \rho) \big(H(g^{-1}k)\big)}~dk,\:\:\:\:\:\:\la \in \mathfrak{a}_\C^*,
\ee
is Harish Chandra's elementary spherical function.
We now list down some well-known properties of the elementary spherical functions which are important for us (\cite{GV}, Prop 3.1.4 and Chapter 4, \S 4.6; \cite{H1}, Lemma 1.18, P. 221).
\begin{lem} \label{lem-phi}
\begin{enumerate}
\item[(1)] $\phi_\la(g)$ is $K$-biinvariant in $g\in G$ and $W$-invariant in $\la\in \mathfrak{a}_\C^*$.
\item[(2)] $\phi_\la(g)$ is $C^\infty$ in $g\in G$ and holomorphic in $\la\in \mathfrak{a}_\C^*$.
\item[(3)] For all $\la\in \overline{\mathfrak{a}^\ast_+}$ and $g\in G$ we have $ |\phi_\la(g)| \leq  \phi_0(g)\leq 1$.
\item[(4)] For $\la\in \mathfrak{a}^*$, the function $\phi_\lambda$ satisfies
$\Delta(\phi_{\la})=-(|\la|^2+|\rho|^2)\phi_{\la}$.
\end{enumerate}
\end{lem}
It follows from $W$-invariance of $\phi_{\la}$, that $\widehat{f}$ is also $W$-invariant. For $K$-biinvariant $L^p$ functions on $G$ the following Fourier inversion formula is well-known (\cite{ST}, Theorem 3.3 and \cite{NPP}, Theorem 5.4): if $f \in L^p(G//K)$, $1\leq p\leq 2$ with $\hat{f}\in L^1(\mathfrak{a}^* , |{\bf c}(\la)|^{-2}~d\la )$ then for almost every $g\in G$,
\be
f(g)=|W|^{-1}\int_{\frak{a}^*}\widehat{f}(\la )\phi_{-\la}(g)|{\bf c}(\la)|^{-2}~d\la.\label{FI}
\ee
Here ${\bf c}(\la)$ denotes Harish Chandra's ${\bf c}$-function and $|W|$ is the number of elements in the Weyl group. Moreover, $f \mapsto \widehat{f}$ is an isometry of $L^2(G//K)$ onto $L^2(\mathfrak{a}^\ast, |{\bf c}(\la)|^{-2}~d\la)^W$, the subspace of $W$-invariant functions in $L^2(\mathfrak{a}^\ast, |{\bf c}(\la)|^{-2}~d\la)$ \cite[Theorem 1.5]{H1}. It is known that \cite[Ch IV, Prop. 7.2]{H2} , there exists a positive number $C$ such that
\be \label{est-c}
|{\bf c}(\lambda)|^{-2}\leq C(1+|\lambda|)^{\dim \mathfrak n},\:\: \textit{ for } \lambda\in \mathfrak a_+^\ast.
\ee
The standard argument using interpolation produces the following Hausdorff-Young inequality: for $p\in [1,2]$ there exists a positive constant $C_p$ such that for all $f\in L^p(G//K)$
\be
\|\widehat{f}\|_{L^{p'}(\mathfrak{a}^\ast,~ |{\bf c}(\la)|^{-2}~d\la)}\leq C_p \|f\|_{L^p(G)}.\nonumber
\ee

Let ${\bf D}(G/K)$ denote the algebra of differential operators on $G/K$ which are invariant under the action of $G$. We also consider the algebra ${\bf D}(A)$ of differential operators on $A$ which are invariant under all translations ($A$ is abelian, thus ${\bf D}(A)$
contains the differential operators on $A$ with constant coefficients). Let ${\bf D}_W(A)\subset {\bf D}(A)$ denote the subalgebra  of invariant operators under the action of $W$ on $A$.  For $D\in {\bf D}(G/K)$, let $R_N(D)$ denote the radial part of $D$ under the action of $N$ on $G/K$ with transversal
manifold $A\cdot o$.
The Harish-Chandra homomorphism $\Gamma: {\bf D}(G/K)\rightarrow {\bf D}_W(A)$ is defined by
\bes
\Gamma(D)=e^{-\rho} R_N(D)\circ e^{\rho}.
\ees Let $\mathcal S(\mathfrak a)$ be the symmetric algebra over $\mathfrak a$ which is defined as the algebra of complex-valued polynomials funcions on the dual space $\mathfrak a^\ast$. Let $\mathcal S(\mathfrak a)^W$ be the
subalgebra of $W$-invariant elements in $\mathcal S(\mathfrak a)$. We identify ${\bf D}_W(A)$  with $\mathcal S(\mathfrak a)^W$ by \cite[Theorem 4.3, p.280]{H2}. Then we have the following fact \cite[Theorem 5.18, p.306]{H2}.
\begin{thm} \label{H-C-X}
The Harish-Chandra homomorphism $\Gamma$ is an isomorphism from ${\bf D}(G/K)$ onto $\mathcal S(\mathfrak a)^W$.
\end{thm}

We also need the following characterization of joint eigenfunctions for $D\in {\bf D}(G/K)$ \cite[Ch II, Lemma 5.15]{H2}.
\begin{lem}	\label{lem-jt-eigen}
For each $\la \in \mathfrak a_{\C}^\ast$,  the spherical functions $\phi_\la$ are joint eigenfunctions for all differential operators in ${\bf D}(G/K)$. More precisely, $\phi_\la$ satisfies the differential equation
\bes
D\phi_\la= \Gamma(D)(i\la) \phi_\lambda.
\ees
\end{lem}

\begin{proof}[Proof of Theorem \ref{thm-cher-X}]
Suppose $f\in C^{\infty}(G//K)$ satisfies the hypothesis (\ref{Carl-cond-X}) for $p\in [1, 2]$ and $Df(x_0)=0$, for all $D\in {\bf D}(G/K)$. We define a measure $\mu$ on $\frak a^\ast$ by
\bes
\mu(E)=\int_E|\hat{f}(\la)|~|\phi_\lambda(x_0)|~|{\bf c}(\la )|^{-2}d\la,
\ees
for all Borel subsets $E$ of $\frak a^\ast$. We note that $\phi_{\la}(x_0)$ is nonzero for almost every $\la\in\frak a^*$ because of analyticity of the function $\la\mapsto \phi_{\la}(x_0)$. As in the proof of Theorem \ref{thm-cher-Rn} we first show that $\mu$ is a finite measure. Indeed, for $p\in (1, 2]$, using the fact that $|\phi_\lambda(x_0)|\leq 1$, for all $\lambda \in \mathfrak a^\ast$, H\"older's inequality, the estimate (\ref{est-c}) and Hausdorff-Young's inequality we have for large $r\in \N$
\beas
&& \int_{\mathfrak a^\ast} |\widehat f(\lambda)|~|\phi_\lambda(x_0)|~|{\bf c}(\lambda)|^{-2}~d\lambda \\
&\leq& \left(\int_{\mathfrak{a}^\ast}(|\la|^2+|\rho|^2)^{p^\prime r}~|\widehat f(\la)|^{p^\prime}~|{\bf c}(\la)|^{-2}~d\la \right)^{\frac{1}{p^\prime}}~ \left(\int_{\mathfrak{a}^\ast} \frac{|{\bf c}(\la)|^{-2}}{(|\lambda|^2+|\rho|^2)^{pr}}~d\la\right)^{\frac{1}{p}}\\
&\leq& A_{r, p} \|\Delta^r f\|_p< \infty,
\eeas
where
\bes
A_{r, p}=\left(\int_{\mathfrak{a}^\ast} \frac{|{\bf c}(\la)|^{-2}}{(|\lambda |^2+|\rho|^2)^{pr}}~d\la\right)^{\frac{1}{p}}.
\ees
With obvious modification this is true for $p=1$ also. We now define the moment sequence $S_j(m)$ as in (\ref{defn-Sj}) for $j\in \{1, \cdots, l\}$ by
\be \label{defn-Sj-X}
S_j(m)=\int_{\mathfrak a^\ast} |\la(\xi_j)|^m~d\mu(\la),
\ee
where $\{\xi_1, \cdots, \xi_l\}$ is an orthonormal basis of $\mathfrak a$. We show that the sequence $S_j(2m)$ satisfies the Carleman condition (\ref{carlcond}) for all $j\in\{1, \cdots, l\}$. Indeed, as above we choose $r\in \N$ sufficiently large such that
\bes
S_j(2m)\leq \int_{\mathfrak{a}^\ast}(|\la|^2+|\rho|^2)^m~|\widehat f(\la)|~|\phi_\lambda(x_0)|~|{\bf c}(\la)|^{-2}~d\la\leq A_{r,p} \|\Delta^{m+r} f\|_p.
\ees
Therefore
\bes
|S_j(2m)|^{-\frac{1}{2m}}\geq  A_{r,p}^{-\frac{1}{2m}}\|\Delta^{(m+r)}f\|_p^{-\frac{1}{2m}}= A_{r,p}^{-\frac{1}{2m}} ~ \left(\|\Delta^{m+r}f\|_p^{-\frac{1}{2(m+r)}}\right)^{\left(1+\frac{r}{m}\right)}.
\ees
Since, $\lim_{m\ra\infty} A_{r,p}^{-\frac{1}{2m}}=1$, it follows from \cite[Lemma 3.3]{BPR} and the hypothesis (\ref{Carl-cond-X}) that
\bes
\sum_{m=1}^{\infty} S_j(2m)^{-\frac{1}{2m}}=\infty.
\ees
Lemma \ref{lempolydense} now implies that the polynomials form a dense subspace of $L^1(\mathfrak a^\ast, \mu)$.  Hence, the set of $W$-invariant polynomials are dense in 
\be
L^1(\mathfrak a^\ast, d\mu)^W= \{f\in L^1(\mathfrak a^\ast, d\mu): f \textit{ is $W$-invariant}\}.\nonumber
\ee
Since $\widehat{Df}\in L^1(\mathfrak a^\ast, |{\bf c}(\lambda)|^{-2}~d\lambda)$, for all $D\in {\bf D}(G/K)$, it follows by the Fourier inversion (\ref{FI})  that 
\bes
D f(x) = |W|^{-1} \int_{\mathfrak a^*} D \phi_\lambda(x)~\widehat f(\la) ~ |{\bf c}(\la)|^{-2} d\la, \;\:\:\: x\in G/K.
\ees
Therefore, the hypothesis $D f(x_0)=0$, for all $D\in {\bf D}(G/K)$ implies that
\bes
\int_{\mathfrak a^*} D \phi_\lambda(x_0)~\widehat f(\la) ~ |{\bf c}(\la)|^{-2} d\la=0.
\ees
Consequently, using Theorem \ref{H-C-X}, and Lemma \ref{lem-jt-eigen} it follows that for all $W$-invariant polynomials $P$
\be \label{F-P-zero}
\int_{\mathfrak a^*} P(\lambda)~\phi_\lambda(x_0)~\widehat f(\la) ~|{\bf c}(\la)|^{-2} d\la=0.
\ee
We observe that $\widehat f\in L^{p^\prime}\cap L^1(\mathfrak a^\ast, |{\bf c}(\lambda)|^{-2}~d\lambda)^W$ and hence is in $L^2(\mathfrak a^\ast, |{\bf c}(\lambda)|^{-2}~d\lambda)^W$.
Since the function $\lambda \mapsto \phi_\lambda(x_0)$ is bounded and $W$-invariant, it follows that the function
\be
F(\lambda)=\overline{\widehat f(\lambda)}~\overline{\phi_{\lambda}(x_0)} \in L^1(\mathfrak a^\ast, d\mu)^W.\nonumber
\ee
Hence, we can approximate $F$ by $W$-invariant polynomials, that is, given any positive number $\epsilon$ there exists a $W$-invariant polynomial $P_\epsilon$ such that
\bes
\|F- P_\epsilon\|_{L^1(\mathfrak a^\ast, d\mu)}< \epsilon.
\ees
Therefore, using (\ref{F-P-zero}) and the above inequality we get that
\beas
&& \int_{\mathfrak a^\ast} |F(\lambda)|^2~|{\bf c}(\la)|^{-2}d\la\\
&=& \left| \int_{\mathfrak a^\ast}\left(F(\lambda)-P_\epsilon(\la)+P_\epsilon(\la)\right)~\widehat f(\la)~\phi_\lambda(x_0)~|{\bf c}(\la)|^{-2}d\la \right|\\
&\leq & \int_{\mathfrak a^\ast}|F(\lambda)-P_\epsilon(\la)|~d\mu(\lambda) + \left|\int_{\mathfrak a^\ast} P_\epsilon(\la)~\widehat{f}(\la)~\phi_\lambda(x_0)~|{\bf c}(\la)|^{-2}d\la \right| < \epsilon,
\eeas
the second integral being zero. It follows that $F$ is the zero function. Since $\lambda\mapsto \phi_\lambda(x_0)$ is real analytic, we conclude that $\widehat f$ vanishes almost everywhere and hence so does $f$.
\end{proof}

The above proof suggests that we can improve Theorem \ref{thm-cher-X-weaker} by replacing the Carleman condition (\ref{Carl-cond-X-weak}) by (\ref{Carl-cond-X}) for any $p\in [1, 2]$.
\begin{cor} \label{Cor-X}
Let $p\in [1, 2]$. Suppose $f\in C^\infty(G/K)$ is such that $\Delta^m f\in L^p(G/K)$, for all $m\in \N\cup \{0\}$ and the sequence $\|\Delta^m f\|_p$ satisfies the Carleman condition (\ref{Carl-cond-X}).
If $f$ vanishes on a nonempty open set in $G/K$ then $f$ vanishes identically.
\end{cor}
\begin{proof}
By the Step 1 of the proof of Theorem 1.3 in \cite{BPR} we reduce the problem to the case of $K$-biinvariant function. Hence the result follows from Theorem \ref{thm-cher-X}.
\end{proof}

\begin{rem}\label{counter1}
\begin{enumerate}
\item It is not hard to see that Theorem \ref{thm-cher-X} fails for $p>2$. Precisely, we choose a $\lambda\in\frak a^*\setminus \{0\}$, and let $x_0\in G/K$ be such that $\phi_{\lambda}(x_0)$ is zero. It is well-known that $\phi_{\lambda}\in L^p(G//K)$, for $p>2$, and for all $D\in {\bf D}(G/K)$
\be
D\phi_{\lambda}(x_0)=\Gamma (D)(i\la )\phi_{\la}(x_0)=0.\nonumber
\ee
Finally,
\be
\sum_{m=1}^{\infty}\|\Delta^m \phi_{\lambda}\|_p^{-\frac{1}{2m}}=(|\lambda|^2+|\rho |^2)^{-\frac{1}{2}}\sum_{m=1}^{\infty}\|\phi_{\lambda}\|_p^{-\frac{1}{2m}}=\infty.
\ee
\item Coming back to Euclidean spaces, it has already been mentioned in the introduction that Theorem \ref{thm-cher-Rn} is not true under the weaker assumption $\Delta_{\R^n}^mf(x_0)=0$. But it is not hard to see from the proof above that an analogous argument can be employed to prove the following: suppose $f\in C^{\infty}(\R^n)$ is a radial function with $\Delta_{\R^n}^m f\in L^p(\R^n)$, for some $p\in [1,2]$, and for all $m\in \N\cup \{0\}$. If $f$ satisfies (\ref{Carl-cond-Rn}) and for some $x_0\in\R^n$, $\Delta_{\R^n}^mf(x_0)=0$, for all $m\in\N\cup \{0\}$ then $f$ vanishes identically. The main reason being that any radial polynomial $P(\la)$ must be a polynomial in $|\la|^2$.
    
However, this statement is false if $p>\frac{2n}{n-1}$. Indeed, for $\lambda\in (0,\infty)$ we consider the radial eigenfunctions $\phi_{\la}$ of $\Delta_{\R^n}$ with eigenvalue $-\la^2$, given by
    \be
    \phi_{\la}(x)=\int_{S^{n-1}}e^{i\la x\cdot\omega}d\sigma \omega,\:\:\:\:x\in\R^n,\nonumber
    \ee
    where $\sigma$ is the normalized rotation invariant measure on the unit sphere $S^{n-1}$. It is well-known that $\phi_{\la}$ satisfies the following estimate \cite[P. 348]{S}.
    \be
    |\phi_{\la}(x)|\leq C(1+|x|)^{-\frac{n-1}{2}},\:\:\:\: x\in\R^n,\nonumber
    \ee
    and hence $\phi_{\la}\in L^p(\R^n)$, for $p> \frac{2n}{n-1}$. It is clear that $\phi_{\la}$ satisfies (\ref{Carl-cond-Rn}). If $x_0$ is a zero of $\phi_{\la}$ then it is evident that $\Delta_{\R^n}^m\phi_{\la}(x_0)$ vanishes for all $m\in\N\cup \{0\}$. It is again an open question whether this version is true for the range $2<p\leq \frac{2n}{n-1}$. 
\end{enumerate}
\end{rem}
\section{Chernoff's theorem for compact symmetric spaces}
We now consider compact Riemannian symmetric space $U/K$ , where $U$ is a connected, simply connected, compact, semisimple Lie group which acts isometrically on $U/K$, and $K$ is a closed subgroup with the property that $U_0^\theta \subset K \subset U^\theta$ for an involution $\theta$ of $U$. Here $U^\theta$ denotes the subgroup of $\theta$-fixed points, and $U_0^\theta$ its identity component.

Let $\mathfrak u$ denote the Lie algebra of $U$ and $\mathfrak u= \mathfrak k \oplus \mathfrak q$ be the Cartan decomposition associated with the involution $\theta$. Then $\mathfrak k$ is the Lie algebra of $K$ and $\mathfrak q$ can be identified with the tangent space $T_o(U/K)$ at the origin $o$. Let $\langle \cdot, \cdot \rangle$ be the inner product on $\mathfrak u$ induced from the Killing form. We assume that the Riemannian metric of $U/K$ is normalized such that it agrees with $\langle \cdot, \cdot\rangle$ on the tangent space $\mathfrak q = T_o(U/K)$. The inner product on $u$ determines an inner product on the dual space $\mathfrak u^\ast$ in a canonical fashion. Furthermore, these inner products have complex bilinear extensions to the complexifications $\mathfrak u_\C$ and $\mathfrak u_\C^\ast$. All these bilinear forms are denoted by the same symbol $ \langle \cdot,\cdot \rangle$.

Let $\mathfrak a \subset \mathfrak q$ be a maximal abelian subspace, $\mathfrak a^*$ its dual space, and $\mathfrak a^*_\C$ the complexified dual space. We assume that $\dim \mathfrak a= l$, called real rank of $U/K$. Let $\Sigma$ denote the set of non-zero (restricted) roots of $\mathfrak u$ with respect to $\mathfrak a$. Then $\Sigma\subset \mathfrak a_\C^\ast$ and all the elements of $\Sigma$ are purely imaginary on $\mathfrak a$. The corresponding Weyl group, generated by the reflections in the roots, is denoted $W$. We make fixed choice of a positive system $\Sigma^+$ for $\Sigma$ and define $\rho \in i \mathfrak a^*$ to be the half sum of the roots in $\Sigma^+$, counted with multiplicity.

We now recall the local Fourier theory for $U/K$ based on elementary representation theory. An irreducible unitary representation $\pi$ of $U$ is said to be a $K$-spherical representation if there exists a non-zero $K$-fixed vector $e_\pi$ in the representation space $V_\pi$. The vector $e_\pi$ (if it exists) is unique up to multiplication by scalars. The following parametrization of $K$-spherical irreducible representations of $U$ is due to Helgason (see \cite{H2}, p. 535).
\begin{thm}
The map $\pi \mapsto \mu$, where $\mu\in i\mathfrak a^\ast$ is the highest weight of
$\pi$, induces a bijection between the set of equivalence classes of irreducible
$K$-spherical representations of $U$ and the set
\be \label{defn-Lambda}
\Lambda^{+}(U/K)=\left\{\mu\in i\mathfrak a^\ast: \frac{\langle \mu, \alpha\rangle}{\langle \alpha, \alpha \rangle}\in \mathbb Z^+, \textit{ for all } \alpha\in \Sigma^+\right\}.
\ee
\end{thm}
Here $\mathbb Z^+=\{0, 1, 2, \cdots\}$.
For each $\mu\in \Lambda^+(U/K)$, we fix an irreducible unitary spherical representation $(\pi_\mu, V_\mu )$ of $U$ and a unit $K$-fixed vector $e_\mu \in V_\mu$. The spherical function on $U/K$ associated with $\mu$ is the matrix coefficient
\be \label{psidefn}
\psi_\mu(u) = \langle \pi_\mu(u)e_\mu, e_\mu \rangle, \:\: u\in U,
\ee
viewed as a function on $U/K$. It is $K$-biinvariant, that is, $K$-invariants on both sides as a function on $U$, and it is independent of the choice of the unit vector $e_\mu$. Henceforth, we shall denote the set of $K$-biinvariant functions in $L^1(U/K)$ by $L^1(U//K)$.
The spherical Fourier transform of a continuous $K$-invariant function $f$ on
$X= U/K$ is the function $\widetilde f$ on $\Lambda^+(U/K)$ defined by
\bes
\widetilde f(\mu)= \int_X f(x)~ \overline{\psi_\mu(x)}~dx.
\ees
where $dx$ is the Riemannian measure on $X$, normalized with total measure $1$. The spherical Fourier series for $f$ is
the series given by
\be \label{FI-com}
f(x)=\sum_{\mu\in \Lambda^+(U/K)} d(\mu)~\widetilde f(\mu)~\psi_\mu(x),
\ee
where $d(\mu) = \dim V_\mu$. The Fourier series converges in $L^2$. If $f$ is smooth, this converges absolutely
and uniformly \cite[Theorem 4.3, p.538]{H2}.

We recall the following fact \cite[Lemma 2.5]{BOP}: for $\mu\in \Lambda^+(U/K)$ the spherical function $\psi_\mu$ extends as a holomorphic function to $U_\C$ and  $\psi_\mu\vline_G=\phi_{\mu+\rho}$, where $G/K$ is the noncompact dual of $U/K$ and $\phi_\lambda$'s are the elementary spherical functions on $G/K$. Consequently, using the $W$-invariance of the function $\mu \mapsto \phi_\mu$ it follows that the function $F(\mu)=\tilde f(\mu-\rho)$ is $W$-invariant if $\mu-\rho \in \Lambda^+(U/K)$.

Let ${\bf D}(U/K)$ denote the algebra of $U$-invariant differential operators on $U/K$. We recall that the Harish-Chandra homomorphism maps
$\Gamma: {\bf D}(U/K) \ra \mathcal S(\mathfrak a)^W$ (see \cite[section 5, p. 207]{OS} for the definition). The following result says that this map is surjective.
\begin{lem}[\cite{OS}, Lemma 5.1] \label{lem-Harish-com}
The Harish-Chandra map $\Gamma$ is an isomorphism onto $\mathcal S(\mathfrak a)^W$.
\end{lem}
The spherical function $\psi_\mu$ satisfies the joint eigen equation \cite[Eq. 5.1]{OS}
\be \label{joint-eigen-com}
D\psi_\mu = \Gamma(D)(\mu + \rho)\psi_\mu, \:\: \:\: D \in {\bf D}(U/K).
\ee
In particular, the Laplace-Beltrami operator $\tilde \Delta$ on $U/K$ belongs to ${\bf D}(U/K)$, and we have $\Gamma(\tilde \Delta)(\mu)= \langle \mu, \mu \rangle- \langle \rho, \rho \rangle$. Since $\tilde \Delta$ is self-adjoint it follows from (\ref{joint-eigen-com}) that
\bes
(\tilde \Delta f)^{\widetilde{}}(\mu) =\left(\langle \mu+\rho, \mu+\rho\rangle-\langle \rho, \rho \rangle \right) \widetilde f(\mu), \:\: \textit{ for all } f\in C^\infty(U//K).
\ees

\begin{proof}[Proof of Theorem \ref{thm-cher-com}]
Since $U/K$ is a finite measure space it suffices to work under the assumption (\ref{Carl-cond-com}) for $p=1$. Let $f\in C^\infty(U//K)$ satisfy the hypothesis of Theorem \ref{thm-cher-com}.
We define the measures $\mu_f$ and $\nu_f$ on the Borel subsets of $i\mathfrak a^\ast$ by
\bes
\mu_f(E)=\sum_{\mu-\rho\in E\cap \Lambda^+(U/K)} d(\mu-\rho)~|\widetilde f(\mu-\rho)|, \:\:\:\:
 \nu_f(E)=\sum_{\mu\in E \cap \Lambda^+(U/K)} d(\mu)~\widetilde f(\mu).
\ees
For $j\in \{1, \cdots, l\}$, we now define the moment sequence $S_j(m)$ as in (\ref{defn-Sj-X}) for the measure $\mu_f$. Then
\be \label{Sj-est}
S_j(2m)\leq \int_{i\mathfrak a^\ast} |\langle \lambda, \lambda \rangle|^m~d\mu_f(\lambda)= \sum_{\mu\in\Lambda^+(U/K)}|\langle \mu+\rho, \mu+\rho \rangle|^m d(\mu)~|\widetilde f(\mu)|.
\ee
We claim that there exists $C>0$ such that   for all nonzero $\mu\in \Lambda^+(U/K)$
\be \label{claim}
|\langle \mu+\rho, \mu+\rho \rangle|\leq C|\langle \mu+\rho, \mu+\rho \rangle-\langle \rho, \rho \rangle|.
\ee
Since $\mu$ and $\rho$ are in $i\mathfrak a^\ast$, we write them as $i\mu^\prime$ and $i\rho^\prime$ respectively, for some $\mu^\prime$ and $\rho^\prime$ in $\mathfrak a^\ast$. Let
\be
\eta= \min\{|\langle \mu, \mu \rangle|: \mu\neq 0, \mu\in \Lambda^+(U/K)\}.\nonumber
\ee
From the definition (\ref{defn-Lambda}) of $\Lambda^+(U/K)$ it follows that $\eta>0$. The definition (\ref{defn-Lambda}) also implies that $\langle \mu^\prime, \rho^\prime \rangle \geq 0$.  Therefore , for all nonzero $\mu\in \Lambda^+(U/K)$
\beas
|\langle \mu+\rho, \mu+\rho \rangle| &=& \langle\mu^\prime, \mu^\prime \rangle+2\langle \mu^\prime, \rho^\prime\rangle+ \langle\rho^\prime, \rho^\prime\rangle \\
&\leq& \langle\mu^\prime, \mu^\prime \rangle+2\langle \mu^\prime, \rho^\prime\rangle+ \langle\rho^\prime, \rho^\prime\rangle \eta^{-1} \langle\mu^\prime, \mu^\prime \rangle\\
&\leq & \left(1+ \langle\rho^\prime, \rho^\prime \rangle \eta^{-1}\right)\left(\langle\mu^\prime, \mu^\prime \rangle+2\langle \mu^\prime, \rho^\prime\rangle \right).
\eeas
This proves the claim (\ref{claim}) with $C=\left(1+ \langle\rho^\prime, \rho^\prime \rangle \eta^{-1}\right)$. We note that $d(\lambda)$ is of polynomial growth \cite[Theorem 9.10, p. 321]{H1}. We choose $r\in N$ sufficiently large. Then using (\ref{claim}) it follows from (\ref{Sj-est}) that
\beas
S_j(2m)&\leq & C^m\sum_{\mu\in\Lambda^+(U/K), \mu\neq 0 }~|\langle \mu+\rho, \mu+\rho \rangle-\langle \rho, \rho \rangle|^m~ d(\mu)~|\widetilde f(\mu)|\\
&\leq& C^m \sum_{\mu\in\Lambda^+(U/K), \mu\neq 0} |(\tilde\Delta f)^{\widetilde{}} (\mu)|^{m+r}~ |\langle \mu+\rho, \mu+\rho \rangle-\langle \rho, \rho \rangle|^{-r}~d(\mu)\\
&=& C^m  \sup_{\mu\in\Lambda^+(U/K)} \left\{|(\tilde\Delta f)^{\widetilde{}} (\mu)|^{m+r}\right\}~ \sum_{\mu\in\Lambda^+(U/K), \mu \neq 0}|\langle \mu+\rho, \mu+\rho \rangle-\langle \rho, \rho \rangle|^{-r}~d(\mu)\\
&\leq& C^m~A_r \|\tilde\Delta^{m+r} f\|_1.
\eeas
Here $A_r$ is the series in the right-hand side of the second last equality above. This is finite since $r$ is sufficiently large and $d(\mu)$ is of polynomial growth.
We now argue as in the proof of Theorem \ref{thm-cher-X} that under the hypothesis (\ref{Carl-cond-com}), the sequence $S_j(2m)$ satisfies the Carleman condition (\ref{carlcond}). Consequently, Lemma \ref{lempolydense} now implies that the polynomials in $i\mathfrak a^\ast$ form a dense subspace of $L^1(\mathfrak ia^\ast, d\mu_f)$.  Hence, the set of $W$-invariant polynomials are dense in $L^{1}(\mathfrak ia^\ast, d\mu_f)^W$. Here $L^{1}(i\mathfrak a^\ast, d\mu_f)^W$ is the set of $W$-invariant functions in $L^1(i\mathfrak a^\ast, d\mu_f)$. By the Fourier inversion formula (\ref{FI-com})  it follows that
\bes
Df(x) =\sum_{\mu\in \Lambda^+(U/K)} d(\mu)~D \psi_\mu(x)~ \widetilde f(\mu) =\int_{i\mathfrak a^\ast} D \psi_\lambda(x)~ d\nu_f(\lambda).
\ees
Therefore, using the hypothesis that $D f(o)=0$, for all $D\in {\bf D}(U/K)$, Lemma \ref{lem-Harish-com} and equation (\ref{joint-eigen-com}) we have that for all $W$-invariant polynomials $P$ on $i\mathfrak a^\ast$
\be\label{eqn-com}
\int_{i\mathfrak a^\ast} P(\la+\rho)~d\nu_f(\la)=0.
\ee
Let $\mu-\rho\in \Lambda^+(U/K)$. Then the function $F(\mu)=\overline{\widetilde f}(\mu-\rho)$ satisfies $F(w\mu)=F(\mu)$, for all $w\in W$. Hence, $F\in L^1(i\mathfrak a^\ast, d\mu_f)^W$. Then   we can approximate $F$ by $W$-invariant polynomials, that is, given $\epsilon>0$ small there exists $P_\epsilon$ such that $
\|F- P_\epsilon\|_{L^1(i\mathfrak a^\ast, d\mu_f)}< \epsilon$.
Therefore, by (\ref{eqn-com}) it follows that
\beas
\sum_{\mu\in \Lambda^+(U/K)} d(\mu)~|\widetilde f(\mu)|^2&=&
 \int_{i\mathfrak a^\ast} F(\lambda+\rho)~d\nu_f(\la)\\
&=& \left| \int_{i\mathfrak a^\ast}\left(F(\lambda+\rho)-P_\epsilon(\lambda+\rho)+P_\epsilon(\lambda+\rho)\right)~d\nu_f(\lambda)\right|\\
&\leq & \int_{i\mathfrak a^\ast}|F(\lambda)-P_\epsilon(\lambda)|~d\mu_f(\lambda) + \left|\int_{i\mathfrak a^\ast} P_\epsilon(\lambda+\rho) d\nu_f(\lambda) \right| < \epsilon.
\eeas
It follows that $\widetilde f$ is zero and hence so is $f$.
\end{proof}

If we assume the vanishing condition of the function on a nonempty open set instead of a single point then Theorem \ref{thm-cher-com} can be extended to smooth functions on $U/K$, which are not necessarily $K$-biinvariant. 
\begin{thm} \label{thm-com-open}
Suppose $f\in C^\infty(U/K)$ satisfies the condition (\ref{Carl-cond-com}), for some $p\in [1,\infty]$. If $f$ vanishes on nonempty open set in $U/K$ then f vanishes identically.
\end{thm}
As in the proof of Corollary \ref{Cor-X}, we can deduce the proof of the theorem above to the $K$-biinvariant functions. For $f\in L^1(U/K)$, we define the $K$-biinvariant component $\mathcal Sf$ of $f$ by the integral
\be \label{radialization}
\mathcal Sf(x) = \int_Kf(kx)~dk, \:\: x\in U/K,
\ee
and for $g \in U$, we define the left translation operator $l_g$ on $L^1(U/K)$ by 
\bes
l_g f(x)= f(gx), \:\:  x\in U/K.
\ees
For a nonzero integrable function $f$, its $K$-biinvariant component $\mathcal S(f)$ may be zero. However, the following lemma shows that there always exists $g\in U$ such that $\mathcal S(l_gf)$ is nonzero. In the case of noncompact symmetric spaces $G/K$ the proof is given in \cite[Lemma 4.6]{BR}.  Let  ${\mathcal B}(o, r)$ denote the open ball of radius $r$ centered at $o$.
\begin{lem} \label{lem-com}
If $f\in L^1(U/K)$ is nonzero then for every $r$ positive there exists $g\in U$ with $gK\in {\mathcal B}(o, r)$ such that $\mathcal S(l_gf)$ is nonzero.
\end{lem}
\begin{proof}
 Suppose the result is false. Then there exists a positive number $r$ such that for all $gK\in {\mathcal B}(o, r)$ the function  $\mathcal S(l_gf)$ is zero.  Hence, for all $t$ positive we have
\bes
\int_U \mathcal S(l_gf)(x) ~ \gamma_t(x^{-1}) ~ dx =0.
\ees
Here $\gamma_t$ is the heat kernel on $U/K$ (see \cite[p. 443]{Thangavelu-2007}). This implies that $(f*\gamma_t)(gK)$ is zero for all positive number $t$. 
That is, $f*\gamma_t$ vanishes on the open ball ${\mathcal B}(o, r)$, for all $t$ positive. Since $f\ast \gamma_t$ is real analytic  on $U$ (in fact holomorphic on $U_\C$) (see \cite[p. 443]{Thangavelu-2007}), it follows that $f\ast \gamma_t$ is the zero function for each $t>0$. Since $\widehat{\gamma_t}(\mu)$ is nonzero for all $\mu\in \Lambda^+(U/K)$, by Fourier inversion formula (\ref{FI-com}), $f$ vanishes identically.  
\end{proof}

\begin{proof}[Proof of Theorem \ref{thm-com-open}]
The analogous proof of Step 1 of the proof of \cite[Theorem 1.3]{BPR} to $U/K$ and Lemma \ref{lem-com} reduce the problem to the case of $K$-biinvariant function on $U$ vanishing on an open set around the origin. The result then follows from Theorem \ref{thm-cher-com}.

\end{proof}

\begin{rem}\label{counter2}
\begin{enumerate}
\item As has been mentioned in the introduction that Theorem \ref{thm-cher-com} may fail if the identity coset $o$ is replaced by some other coset $x_0 K$. To see this we choose a nonzero $\mu\in \Lambda^+(U/K)$ and let $x_0 K$ be such that $\psi_{\mu}(x_0)=0$. Since $\psi_{\la}$ is an eigenfunction of $\tilde \Delta$, it follows that $\tilde \Delta^m\psi_{\mu}(x_0)$ is zero for all
$m\in\N\cup\{0\}$.
Moreover, $\psi_{\mu}$ satisfies (\ref{Carl-cond-com}) as
\be
\sum_{m=1}^{\infty}\|\tilde \Delta^m \psi_{\mu}\|_1^{-\frac{1}{2m}}=|\langle \mu+\rho, \mu+\rho \rangle- \langle \rho, \rho \rangle|^{-\frac{1}{2}}\sum_{m=1}^{\infty}\|\psi_{\mu}\|_1^{-\frac{1}{2m}}=\infty.\nonumber
\ee
\item Bochner and Taylor in \cite[Theorem 10]{BT} proved the following result on quasi-analytic functions on the unit sphere $S^{n-1}$: let $f\in C^\infty(S^{n-1})$ be such that
\be \label{cond-100}
\sum_{m\in \N}\|\Delta_{S^{n-1}}^mf\|_{\infty}^{-\frac{1}{m}}=\infty.
\ee
If $\Delta_{S^{n-1}}^m f(x)=0$ for all $m\in \N$ and $x\in U$, a set of analytic determination then $f$ vanishes identically.

If we assume that $f$ is radial then Theorem \ref{thm-cher-com} improves this result. Indeed, the condition (\ref{cond-100}) implies the hypothsis (\ref{Carl-cond-com}) for $p=\infty$ of Theorem \ref{thm-cher-com} . Hence, $f$ vanishes identically if $\Delta_{S^{n-1}}^m f(o)=0$, for all $m\in \N$ which is a much weaker assumption compared to vanishing on a set of anaytic determination.

\item Analogues of Theorem \ref{thm-cher-X} can also be formulated and proved (using the same technique) in the setting of Dunkl transform \cite{Ro} and of hypergeometric transforms associated with root systems \cite{NPP, OpdamActa}.

\end{enumerate}
\end{rem}


\begin{thebibliography}{99}


\bibitem{BGST} Bagchi, S.; Ganguly, P.; Sarkar, J.; Thangavelu, S.; \textit{On theorems of Chernoff and Ingham on the Heisenberg group}, arXiv:2009.14230.

\bibitem{BPR} Bhowmik, M.; Pusti, S.; Ray, S. K.; \textit{Theorems of Ingham and Chernoff on Riemannian symmetric
spaces of noncompact type}, J. Funct. Anal. vol. 279 (11), 2020, 108760.

\bibitem{BR}Bhowmik, M.; Ray, Swagato K., \textit{A theorem of Levinson for Riemannian symmetric spaces of noncompact type}, International Mathematics Research Notices, (IMRN). vol. 2021 (4), February 2021, pages 2403-2436.

\bibitem{B} Bochner, S.; \textit{Quasi-analytic functions, Laplace operator, positive kernels}, Ann. of Math. (2) 51 (1950), 68-91.

\bibitem{BT} Bochner, S.; Taylor, A. E.; \textit{Some Theorems on Quasi-Analyticity for Functions of Several Variables}, Amer. J. Math. 61 (1939), no. 2, 303-329. 

\bibitem{BOP} Branson, T.; \'Olafsson, G.; Pasquale, A.; \textit{The Paley-Wiener theorem and the local Huygens' principle for compact
symmetric spaces: The even multiplicity case}, Indag. Math. 16 (2005) 393-428.

\bibitem {Ch} Chernoff, Paul R.; \textit{Quasi-analytic vectors and quasi-analytic functions}, Bull. Amer. Math. Soc. 81 (1975), 637-646. 

\bibitem{Dj} de Jeu, Marcel; \textit{Determinate multidimensional measures, the extended Carleman theorem and quasi-analytic weights}, Ann. Probab. 31 (2003), no. 3, 1205-1227. 

\bibitem{DJ2} de Jeu, Marcel; \textit{Subspaces with equal closure}, Constr. Approx. 20 (2004), no. 1, 93-157. 

\bibitem{GV} Gangolli, R.; Varadarajan V. S.; \textit{Harmonic Analysis of Spherical Functions on Real Reductive Groups},  Springer-Verlag, Berlin, 1988. 

\bibitem{GT1} Ganguly, P.; Thangavelu, S.; \textit{An uncertainty principle for some eigenfunction expansions with applications}, arXiv:2011.09940.

\bibitem{GT2} Ganguly, P.; Thangavelu, S.; \textit{An uncertainty principle for spectral projections on rank one symmetric spaces of noncompact type}, arXiv:2011.09942.

\bibitem{H} Helgason, S.; \textit{Differential geometry, Lie groups, and symmetric spaces}, Graduate Studies in Mathematics, 34, American Mathematical Society, Providence, RI, 2001. 

\bibitem{H1} Helgason, S.;\textit{Geometric Analysis on Symmetric Spaces}, Mathematical Surveys and Monographs 39. American Mathematical Society, Providence, RI, 1994. 

\bibitem{H2} Helgason, S.; \textit{Groups and geometric analysis, Integral geometry, invariant differential operators, and spherical functions}, Mathematical Surveys and Monographs, 83. American Mathematical Society, Providence, RI, 2000. 

\bibitem{In} Ingham, A. E.; \textit{A Note on Fourier Transforms.} J. London Math. Soc. 9 (1934), no. 1, 29-32.

\bibitem{NPP} Narayanan, E. K.; Pasquale, A.; Pusti, S.; \textit{Asymptotics of Harish-Chandra expansions, bounded hypergeometric functions associated with root systems, and applications}, Adv. Math. 252 (2014), 227-259.

\bibitem{OS} \'Olafsson, G.;  Schlichtkrull, H.; \textit{ A local Paley-Wiener theorem for compact symmetric spaces}, Adv. Math. {\bf 218} (2008), no. 1, 202-215.

\bibitem{OpdamActa} Opdam, E. M.; \textit{Harmonic analysis for certain representations of graded Hecke algebras}, Acta Math. 175 (1995), no. 1, 75-121. 

\bibitem {Po} Poussin, ch. de la Vall\'ee ; \textit{Quatre le\c{c}ons sur les fonctions quasi-analytiques de variable r\'eelle}, Bull. Soc. Math. France 52 (1924), 175-203.

\bibitem{Ro} R\"{o}sler, Margit; \textit{Dunkl operators: theory and applications. Orthogonal polynomials and special functions},  93-135, Lecture Notes in Math., 1817, Springer, Berlin, 2003.

\bibitem{Ru} Rudin, W.; \textit{Real and complex analysis} Third edition. McGraw-Hill Book Co., New York, 1987. 

\bibitem{S} Stein, E.M.; \textit{Harmonic Analysis: Real-variable methods, Orthogonality, and Oscillatory integrals}, Princeton University Press, Princeton, NJ, 1993.

\bibitem{ST}  Stanton, Robert J.; Tomas, Peter A.; \textit{Pointwise inversion of the spherical transform on $L^p(G/K)$, $1\leq p<2$}, Proc. Amer. Math. Soc. 73 (1979), no. 3, 398-404. 

\bibitem{Thangavelu-2007} Thangavelu, S.; \textit{Holomorphic Sobolev spaces associated to compact symmetric spaces.}  J. Funct. Anal. 251 (2007), no. 2, 438–462.

\end{thebibliography}
\end{document}